\documentclass[11pt,reqno]{amsart}
\usepackage{graphicx,color}
\usepackage{amsmath,mathrsfs,amsfonts,amssymb}
\usepackage{mathtools}
\usepackage{stmaryrd}
\numberwithin{equation}{section}

\newcommand{\p}{\partial}
\renewcommand{\O}{\Omega}
\renewcommand{\div}{\operatorname{div}}
\newtheorem{theorem}{{\bf Theorem}}[section]

\newtheorem{corollary}[theorem]{{\bf Corollary}}
\newtheorem{lemma}[theorem]{{\bf Lemma}}
\newtheorem{prop}[theorem]{{\bf Proposition}}

\newcommand{\tr}{\operatorname{tr}}

\newcommand{\weight}[1]{\langle #1\rangle}

\newcommand{\R}{\mathbb{R}}

\renewcommand{\S}{\mathbb{S}}

\renewcommand{\bar}{\overline}
\renewcommand{\tilde}{\widetilde}

\def\({\left(}
\def\){\right)}
\def\l|{\left|}
\def\r|{\right|}

\newcommand{\diverg}{\operatorname{div}}

\newcommand{\barg}[1]{\bigl(#1\bigr)}

\newcommand{\bset}[1]{\bigl\{#1\bigr\}}

\newcommand{\bnorm}[1]{\bigl\|#1\bigr\|}

\newcommand{\bbar}[1]{\bigl|#1\bigr|}

\newcommand{\scp}[2]{\langle #1, #2\rangle}
\newcommand{\bscp}[2]{\bigl< #1, #2\bigr>}

\newcommand{\dv}[1]{\,{\rm d}#1}
\newcommand{\bref}[1]{(\ref{#1})}

\def\BXzeroR{{B}_{X_0}(0,R)}
\def\BXzeroRone{{B}_{X_0}(0,R_1)}
\def\VWloka{\mathbb{V}}  
\def\HWloka{\mathbb{H}}  
\def\SonetwoWloka{H^{1}}

\def\HNSO{L^2_\sigma(\Omega)}
\def\VNSO{H^1_{0,\sigma}(\Omega)}

\def\VNSprimeO{H^{-1}_{\sigma}(\Omega)}
\def\HNSObb{(\HNSO)}

\def\Zspace{\mathscr{I}}

\def\uh{u_h}
\def\uhi{u_{hi}}
\def\uhone{u_{h1}}
\def\uhtwo{u_{h2}}
\def\Qh{Q_h}
\def\Qhi{Q_{hi}}
\def\Qhone{Q_{h1}}
\def\Qhtwo{Q_{h2}}
\newcommand{\ud}{\mathrm{d}}

\newcommand{\pa}{\partial}

\def\ve{\varepsilon}

\newcommand{\dt}[1]{\partial_t #1} 

\begin{document}

\title[Wellposedness of $Q$-tensor system]{On the initial boundary value problem of a Navier-Stokes/$Q$-tensor
model for liquid crystals  }

\author{Yuning Liu }
\address{Joint Mathematics Institute,   New York University Shanghai, 200122, Shanghai, P. R. China}
\email{yl67@nyu.edu}

\author{Wei Wang}
\address{School of Mathematical Sciences, Zhejiang University, 310027, Hangzhou, P. R. China}
\email{wangw07@zju.edu.cn}

\begin{abstract}
{This work is concerned with the solvability of a Navier-Stokes/$Q$-tensor coupled system modeling the nematic liquid crystal flow on a bounded domain in three dimensional Euclidian space with strong anchoring boundary condition  for the order parameter. We prove the existence of local in time strong solution to the system with the anisotropic elastic energy. The proof is based on mainly two ingredients: first,  we show that the Euler-Lagrange operator corresponding to the Landau-de Gennes free energy with general elastic coefficients fulfills the strong Legendre condition. This result together with a higher order energy estimate leads to the well-posedness of the linearized system, and then a local in time solution of the original system which is regular in temporal variable follows via a fixed point argument. Secondly, the hydrodynamic part of the coupled system can be reformulated into a quasi-stationary Stokes type equation to which the regularity theory of the generalized Stokes system, and then a bootstrap argument can be applied to enhance the spatial regularity of the local in time solution.
}
\end{abstract}

 \maketitle





\section{Introduction}\label{sec:intro}

Nematic liquid crystal is a sort of material which may flow as a conventional liquid while the molecules are oriented in a crystal-like way.
One of the successful continuum theories modeling nematic liquid crystals is the $Q$-tensor theory,   also referred to as Landau-de Gennes theory, which  uses a $3\times 3$ traceless and symmetric matrix-valued function $Q(x)$ as order parameters to characterize the orientation of molecules near material point $x$ (cf. \cite{DeGennesProst1995}).
The matrix $Q$, also called $Q$-tensor, can be interpreted as the second momentum of a number density function
\begin{equation*}
Q(x)=\int_{\mathbb{S}^2} (m\otimes m-\frac13I_3)f(x,m)\ud m,
\end{equation*}
where $f(x, m)$   corresponds to the number density of liquid crystal molecules which orient along the direction $m$ near material point $x$.
The configuration space for $Q$-tensor will be denoted by
\begin{equation}\label{eq:1.08}
  \mathcal{Q}=\{Q\in \R^{3\times 3}\mid Q_{ij}=Q_{ji},~\sum_{i=1}^3Q_{ii}=0\}.
\end{equation}
 If   $Q(x)$ has three equal eigenvalues, it must be zero and this corresponds to the {\it isotropic phase}. When $Q(x)$ has two equal eigenvalues, it can be written as
\begin{equation*}
Q(x)=s(x)\(n(x)\otimes n(x)-\frac 13I_3\),\end{equation*}
for some $n(x)\in \mathbb{S}^2$ and $s(x)\in\mathbb{R}$
and it is said to be    {\it uniaxial}. If all three eigenvalues of $Q(x)$ are  distinct, it is called  {\it biaxial} and  can be written as
\begin{equation*}
Q(x)=s(x)\(n(x)\otimes n(x)-\frac 13I_3\)+r(x)\(m(x)\otimes m(x)-\frac 13I_3\),\end{equation*}
where $n(x),m(x)\in \S^2$  and $s(x),r(x)\in\R$.

The classic Landau-de Gennes theory associates  to each $Q(x)$ a   free energy  of the following form:
\begin{align}\nonumber
  \mathcal{F}( Q ,\nabla Q )=&\int_{\Omega}\( \frac{a}2\tr Q ^2
-\frac{ b}{3}\tr Q ^3+\frac{c}{4}(\mathrm{tr} Q ^2)^2 \)\ud x\\ \nonumber
&+\frac{1}{2}\int_{\Omega} \Big(L_1|\nabla Q |^2+L_2Q_{ij,j}Q_{ik,k}
+L_3Q_{ij,k}Q_{ik,j}+L_4Q_{\ell k}Q_{ij,k}Q_{ij,\ell} \Big)   \ud x\\ \label{eq:1.07}
=: &~ \mathcal{F}_b( Q )+\mathcal{F}_e(\nabla Q,Q).
\end{align}
Here and in the sequel,  we shall adopt the Einstein's summation convention by summing over  repeated greek letters.
In \eqref{eq:1.07},  $ \mathcal{F}_b(Q)$ is the bulk energy,   describing the isotropic-nematic phase
transition while   $\mathcal{F}_e(\nabla Q,Q )$
is the elastic energy which characterizes the distortion effect. The parameters $a, b,c $ are temperature dependent constants
  with $b, c >0$, and $L_1, L_2, L_3, L_4$ are  elastic coefficients.
In the sequel, we shall call $\mathcal{F}(Q,\nabla Q)$  {\it{ isotropic}} if $L_2=L_3=L_4=0$ and
  {\it anisotropic} if at least one of  $L_2, L_3, L_4$ does not vanish. This work is devoted to   the later case.
  Note that the term in \eqref{eq:1.07} corresponding to $L_4$ is  cubic   and will lead to severe  analytic difficulties:
  it is shown in \cite{maiersaupe} that $\mathcal{F}(Q,\nabla Q)$ with $L_4\neq 0$ is not bounded from below.
  In this work,  we follow \cite{MR2960033,wangzz} and  assume:
\begin{equation}\label{eq:2.01}
L_1>0, \quad    L_1+L_2+L_3=: L_0>0, \quad L_4=0.
\end{equation}

In order to introduce the sytem under consideration, we need some notation.
For any $Q\in\mathcal{Q}$ defined by \eqref{eq:1.08},   $\mathcal{S}_Q(M)$ will be a linear  operator acting on   any  $3\times 3$  matrix $M$
\begin{equation}\label{eq:1.15}
  \mathcal{S}_ Q ( M)=\xi\Big( \frac12(M+M^T)\cdot( Q +\frac13I_3)+\frac12( Q +\frac13I_3)\cdot(M+ M^T)-2( Q +\frac13I_3) (Q+\frac13I_3) : M\Big).
\end{equation}
Note that $\mathcal{S}_ Q ( M)$ is a traceless symmetric   matrix and if $M$ is symmetric and traceless additionally, it reduces to
\begin{equation*}
  \mathcal{S}_ Q ( M)=\xi\Big( M\cdot( Q +\frac13I_3)+( Q +\frac13I_3)\cdot M-2( Q +\frac13I_3) (Q : M)\Big).
\end{equation*}
Here $A:B=\tr AB^T$ and $A\cdot B$ denotes the usual matrix product of $A,B$ and the `dot' will be sometimes omitted if it is clear from the context.
The parameter  $\xi$  is a constant depending on the molecular details of a given liquid crystal and measures the ratio between the tumbling and the aligning effect that a shear flow would exert over the liquid crystal directors. Concerning the hydrodynamic part,  for any vector field $u$, its gradient  $\nabla u$ can  be written  as the sum of the symmetric and anti-symmetric parts:
\begin{equation*}
  \nabla u=D(u)+W(u),
\end{equation*}
where
 \begin{equation*}
D(u)=\frac12(\nabla u+(\nabla u)^T),\quad W(u)=\frac12(\nabla u-(\nabla u)^T).
\end{equation*}
With these   notation,  the Navier-Stokes/Q-tensor  system,   proposed in Beris-Edwards \cite{BerisEdwards1994}, can be written as:
\begin{subequations}\label{yuning:be}
\begin{align}
  &u_t+u\cdot\nabla u=\nabla P+\nabla\cdot(\sigma^{s}+\sigma^{a}+\sigma^d),\label{eq:BE-v}\quad \\
&\nabla\cdot u=0, \\
&Q_t+u\cdot\nabla  Q  + Q \cdot W(u)-W(u)\cdot Q ={\Gamma} \mathcal{H}(Q)  +\mathcal{S}_Q (D(u)).\label{eq:BE-Q}
\end{align}
\end{subequations}
In \eqref{eq:BE-Q}, $\Gamma$ is the rotational diffusion constant and without loss of generality, we shall assume $\Gamma=1$ in the sequel.
The unknowns $(u,P)$ correspond to the  velocity/pressure of the hydrodynamics respectively. The stress  terms
$\sigma^{s}$, $\sigma^{a}$ and $\sigma^d$ on the right hand side of \eqref{eq:BE-v} are
symmetric viscous stress, anti-symmetric viscous stress and distortion stress respectively:
\begin{subequations}
  \begin{align}
    \sigma^{s}(u,Q)&:= \nu D(u)-\mathcal{S}_Q( \mathcal{H}(Q) ),\\
    \sigma^{a}(Q)&:= Q \cdot \mathcal{H}(Q) - \mathcal{H}(Q) \cdot Q ,\\
\sigma^d(Q)&:= \left(\sigma^d_{ij}(Q)\right)_{1\leq i,j\leq 3}=-\left(\frac{\partial \mathcal{F}(Q,\nabla Q)}{\partial Q_{k\ell,j}}Q_{k\ell,i}\right)_{1\leq i,j\leq 3}.
  \end{align}
\end{subequations}
$ \mathcal{H}(Q)$ is the molecular field, defined as the variational derivative of \eqref{eq:1.07} and is written as the sum of the bulk part   and the elastic part:
\begin{equation}\label{eq:1.11}
    \mathcal{H}(Q):= -\frac{\delta \mathcal{F}(Q,\nabla Q)}{\delta Q}:=-\mathcal{L}(Q)-\mathcal{J}(Q).
\end{equation}
The operator $\mathcal{L}$ and $\mathcal{J}$ can be written explicitly by
\begin{equation}\label{qequ}
\begin{split}
 -\mathcal{L}_{ij}(Q)
=& L_1\Delta{Q_{ij}} +\frac{L_2+L_3}2(Q_{ik,kj}+Q_{jk,ki}
-\frac 23Q_{\ell k,k\ell}\delta_i^j), \end{split}
\end{equation}
\begin{equation*}
-\mathcal{J}_{ij}(Q)= -aQ_{ij}+b(Q_{jk}Q_{ki}-\frac 13\tr(Q^2)\delta_{ij})-c
\tr(Q^2)Q_{ij} .
\end{equation*}
We note that the operator defined via \eqref{qequ} can not be considered as a perturbation of $L_1\Delta$ as we only assumes \eqref{eq:2.01}. Actually, one of the key results in this work is Lemma \ref{lem:estimate-L} below, showing that  \eqref{qequ} fulfills the strong Legendre condition.

The coupled  system \eqref{yuning:be} has been recently studied by several authors.
For the case $\xi = 0$,  which corresponds to the situation when  the molecules only tumble in a shear flow  but   are not aligned by the flow,
the existence of global weak
solutions to the Cauchy problem in $\R^d$ with $d = 2, 3$ is proved  in \cite{PZ2}. Moreover,  solutions with higher order  regularity
and the weak-strong uniqueness for $d = 2$ is discussed.
Later, these results are generalized in \cite{PZ1} to the case when $|\xi|$ is sufficiently small.
Large time behavior of the solution to the Cauchy problem in
$\R^3$ with $\xi=0$ is recently discussed in \cite{dai2014asymptotic}. The global well-posedness and long-time behavior
of system with nonzero $\xi$ in the two-dimensional periodic setting are studied in \cite{CRWX}.

In \cite{HD}, the authors considered  Beris-Edwards system with anisotropic elastic energy    \eqref{eq:1.07} (with $L_2+L_3>0$ and $L_4=0$).
They proved  the  existence of global weak solutions as well as the existence
of a unique global strong solution for the Cauchy problem in $\R^3$ provided that the fluid viscosity is sufficiently large.
In  \cite{ChenXu2016,IyerXuZarnescu2015}, the  weak solution of  the  gradient
flow generated by the general Landau-de Gennes energy \eqref{eq:1.07} with $L_4\neq 0$ is established for small initial data.

Some recent progresses have also been made on the  analysis of certain modified
versions of Beris-Edwards system. In \cite{wilkinson2015strictly}, when $\xi=0$ and  the polynomial bulk energy is replaced by a singular potential derived from molecular Maier-Saupe theory, the author proved, under periodic boundary conditions,  the existence of global weak solutions in space dimension two and three. Moreover,  the existence and uniqueness of global regular solutions for dimension two is obtained.
In \cite{FeireislEtAlQTensor,feireisl2015nonisothermal}, the authors derived a nonisothermal
variants of  \eqref{yuning:be}    and   proved the existence
of global weak solutions in the case of a singular potential under periodic boundary conditions
for general $\xi$.
We also mention  that a rigorous derivation of the general Ericksen-Leslie system from the
small elastic limit of Beris-Edwards system (with arbitrary $\xi$)  is recently given
in \cite{wangzz}  using the Hilbert expansion.

In the aforementioned works, the domain under consideration is either the whole space or the tori. The initial-boundary value
problems of  \eqref{yuning:be}  have been also investigated by several authors, see for instance
 \cite{AbelsDolzmannLiu2014,abels2016strong,guillen2014weak,guillen2015weak},
in which the existences of weak solutions has been studied. In addition, in \cite{AbelsDolzmannLiu2014,guillen2014weak},
the authors proved the existence of local in time solution with higher order time regularity for \eqref{yuning:be} through different approaches. However,
the higher order spacial regularity is not obtained due to the lack of effective energy estimate in the presence of inhomogenous boundary condition for $Q$.

The main goal of the presented work is to improve the results in \cite{AbelsDolzmannLiu2014,guillen2014weak} to nature regularity in space variable.
This gives a full answer to the construction of local in time strong solution of \eqref{yuning:be} in the presence of inhomogenous boundary condition for $Q$.

We shall consider the initial-boundary conditions
\begin{subequations}\label{yuning:BE-ini-bou}
\begin{align}\label{eq:BE-ini}
&(u,Q)|_{t=0}=(u_0,Q_0),\\
&(u,Q)|_{\p \Omega}=(0,Q_0|_{\p \Omega}),\label{eq:BE-bou}
\end{align}
\end{subequations}
where $Q_0=Q_0(x)$ is   time-independent.
Note that such a result requires a   compatibility condition on the initial data
 $Q_0$. To see that, we write \eqref{yuning:be}  in an
abstract form
 \begin{equation}\label{stanon}
  \frac {d}{dt} (u,Q) =\mathcal{E}(u,Q),
 \end{equation}
where $\mathcal{E}:\VNSO\times H^2(\Omega)\mapsto \VNSprimeO\times L^2(\Omega)$ is defined by
 \begin{equation*}
   \begin{split}
&\bscp{\mathcal{E}(u,Q)}{(\varphi,\Psi)}
=-\int_{\Omega} (-u\otimes u+ \sigma^{s}+\sigma^{a}+\sigma^d):\nabla\varphi\, {\rm d}x
 \\
&+ \int_{\Omega}(-u\cdot\nabla  Q  - Q \cdot W(u)+W(u)\cdot Q + \mathcal{H}(Q)  +\mathcal{S}_Q (D(u))): \Psi\, {\rm d}x,
   \end{split}
 \end{equation*}
for all $(\varphi,\Psi)\in \VNSO\times L^2(\Omega;\mathcal{Q})$. Note that the   functional spaces used here are defined in Section \ref{sec:Notation}.
Since \eqref{eq:BE-bou} specifies a time-independent  boundary condition,
it follows that $\dt{Q}|_{\p \Omega}=0$ which leads to the compatibility
condition that the trace of the second component
on the right-hand side of \eqref{stanon} vanishes on $\p\O$.
This motives to define the admissible class for the initial data
\begin{equation}\label{init}
\Zspace =\bset{  (u_0,Q_0)\in \VNSO\times H^2(\Omega;\mathcal{Q})\mid \mathcal{E}(u_0,Q_0)
\in L^2_{\sigma}(\Omega)\times H^1_{0}(\Omega)
}.
\end{equation}
 It is not hard to see $\Zspace$ is not empty. For example, for any $Q$ solving $\mathcal{H}(Q)=0$ and any $u\in H^2_0(\O)$, we have $(u,Q)\in \Zspace$.
We note that such a compatibility condition is nature in the sense that it can not be disregarded by changing the function spaces   unless one considers very weak solution.

The main result of this paper  can be  stated as follows.
\begin{theorem} \label{thm:localstrong}
Assume the coefficients of elastic energy satisfy (\ref{eq:2.01}). Then for any  $(u_0,Q_0)\in \Zspace$ with  $Q_0|_{\p\O}\in H^{5/2}(\p\O) $, there exists some $T>0$ such that
 the system  \eqref{yuning:be} and \eqref{yuning:BE-ini-bou} has a unique solution \begin{equation}\label{eq:1.27}
\begin{split}
u&\, \in H^2(0,T;H^{-1}(\O))\cap H^1(0,T;H_\sigma^1(\Omega))\cap L^\infty(0,T;H^2(\O))\,,\\
Q&\, \in H^2(0,T;L^2(\Omega;\mathcal{Q}))\cap H^1(0,T;H^2(\Omega;\mathcal{Q}))\cap L^\infty(0,T;H^3(\O;\mathcal{Q})).\,
\end{split}
\end{equation}
\end{theorem}

 Theorem \ref{thm:localstrong} essentially improves the spatial regularity of solution obtained in \cite{AbelsDolzmannLiu2014} and generalizes their result to the case of anisotropic elastic energy.  This is accomplished by the crucial observation
that the  terms  containing third order derivatives on $Q$ in \eqref{eq:BE-v} can be eliminated and  the system can be
reduced into   a Stokes-type system with positive definite  viscosity coefficient. Moreover,   under the general assumption
\eqref{eq:2.01},  the   operator $\mathcal{L}$ defined by \eqref{qequ} is
strongly elliptic. This fact leads to  $W^{2,p}$-estimates for the solution so that we can work with the general case of anisotropic energy
rather than the isotropic energy case ($L_2=L_3=0$). The strong ellipticity of $\mathcal{L}$ is proved in Lemma \ref{lem:estimate-L} by an explicit construction
of the coefficient matrix, which involves a fairly sophisticated anisotropic tensor of order six.
We also mention that the local in time strong solution  constructed here is valid for any $\xi\in\R$.

The rest parts of the work is organized as follows. In Section \ref{sec:Notation}, we introduce
notation and analytic tools that will be used throughout the paper. The most important results
involve  the solvability theorem on the generalized Stokes system, due to Solonnikov
\cite{MR1855442}, as well as Lemma \ref{lem:estimate-L} on the analysis of the operator $\mathcal{L}$.
In Section~\ref{fasetup} an abstract evolution equation that   incorporates \eqref{yuning:be},
\eqref{yuning:BE-ini-bou} and a  compatibility condition is introduced and the functional
analytic framework is established. The core part, Section \ref{section3}, is devoted to the proof
of Theorem \ref{thm:localstrong} by showing that the abstract evolution equation has a local in
time solution. This is accomplished  by proving the existence of a local in time solution that is
regular in temporal variable in the first stage, following the method in \cite{AbelsDolzmannLiu2014},
and then using the structure of \eqref{yuning:be} to eliminate the higher order terms in the additional
stress tensors of  \eqref{eq:BE-v} and recast it  into a generalized Stokes system. Afterwards, the spatial
regularity of \eqref{yuning:be} with initial-boundary condition \eqref{yuning:BE-ini-bou} is improved
using the $L^p$-estimate of the generalized Stokes system together with   bootstrap arguments.


\section{Preliminaries}\label{sec:Notation}

\subsection{Notations} Throughout this paper,   the Einstein's summation convention will be adopted. That is, we shall sum over  repeated greek letters. For any $3\times 3$ martrix $A,B\in\mathbb{R}^{3\times 3}$,   their usual  matrix product will be denoted by $A\cdot B$ or even shortly by $AB$ if it is clear from the context. The Frobenius product  of two matrices corresponds to $A:B= \tr AB^T =  A_{ij}B_{ij}$ and this induces a norm   $|A|=\sqrt{A_{ij}A_{ij}}$.  For any  matrix-valued function $F=(F_{ij})_{1\leq i,j\leq 3}$,  we denote $F_{ij,k}=\partial_kF_{ij}$ and   $\div F=\nabla\cdot F = (\partial_j F_{ij})_{1\leq i\leq 3}$.

In tensor analysis,  the Levi-Civita symbol $\{\ve^{ijk}\}_{1\leq i,j,k\leq 3}$ and Kronecker symbol $\{\delta^i_j\}_{1\leq i,j \leq 3}$  are very useful to deal  with operations involving inner and wedge product:  for any $a,b\in\R^3$, their inner and wedge products are given by
\begin{equation*}
a\cdot b= a_ib_j\delta_i^j,\qquad a\wedge b=(a_jb_k\ve^{ijk})_{1\leq i\leq 3}
\end{equation*}
respectively. The following identity is well known:
 \begin{equation}\label{eq:1.02}
   \ve^{ijk}\ve^{imn}=(\delta^{j}_{m}\delta^{k}_{n}-\delta^{j}_{n}\delta^{k}_m).
 \end{equation}
For any vector field $u$, its divergence and curl can be calculated by
 \begin{equation}\label{eq:1.03}
 \nabla \cdot u=\delta_i^j u_{i,j},\qquad  \nabla\wedge u=(\ve^{ijk}u_{  k,j})_{1\leq i\leq 3}.
 \end{equation}
\subsection{Function spaces and the generalized Stokes system}
Throughout this work,   $\O\subset \R^3$ will be a bounded domain with smooth boundary and $\O_T=\O\times (0,T)$ will denote the parabolic cylinder. Standard notation for the Lebesgue and Sobolev spaces $L^p(\Omega)$ and
$W^{s,p}( \Omega )$ as well as $L^p( \Omega ;M)$ and $W^{s,p}( \Omega ;M)$ for the corresponding spaces for $M$-valued functions will be employed.
Sometimes  the domain and the range is omitted for simplicity if it is clear from the context.
The $L^2$-based Sobolev spaces are denoted by   $H^s( \Omega ;M)$ or simply by $H^s(\O)$.
For any Banach space $\mathscr{H}$,  $\langle \cdot,\cdot\rangle_{\mathscr{H},\mathscr{H}^*}$ denotes the dual product  between $\mathscr{H}$ and its dual space $\mathscr{H}^*$ and we shall simply write   $\langle\cdot,\cdot\rangle$ if the function spaces under consideration are clear from the context.

An important result  related to the incompressible Navier-Stokes equation is   the Helmholtz decomposition
\begin{align*}
 L^2(\Omega;\R^3) = \HNSO \oplus {\HNSObb}^\perp,
\end{align*}
where $L^2_\sigma(\O)$ denotes the space of solenoidal vector field and its orthogonal space is given by
\begin{equation*}
  {\HNSObb}^\perp = \bset{ u \in L^2(\Omega;\R^3),\,u=\nabla q
\text{ for some }q\in H^1(\Omega) }\,.
\end{equation*}
The Helmholtz projection (also referred to as Leray projection), i.e., the orthogonal projection $L^2( \Omega ;\R^3) \mapsto \HNSO$,
is denoted by $P_\sigma$. The readers can refer to \cite{BuchSohr} for its basic properties. For any  $f\in H^{-1}( \Omega ;\R^3)$,   $P_\sigma f \in H^{-1}_{\sigma}( \Omega )$ is interpreted by $P_\sigma f=f|_{H^1_{0,\sigma}( \Omega )}$. Moreover, for $F\in L^2( \Omega ;\R^{3\times 3})$, $\diverg F\in H^{-1}( \Omega ;\R^3)$ is defined by
\begin{equation*}
  \weight{\diverg F, \Phi}_{H^{-1},H^1_0}= -\int_{\Omega} F: \nabla \Phi \dv{x},\qquad \text{for all }\Phi\in H^1_0( \Omega ;\R^3)\, .
\end{equation*}

We end this subsection by  the following result due to Solonnikov \cite{MR1855442}, which is crucial in the discussion of  the spatial regularity during the proof of Theorem \ref{thm:localstrong}.
\begin{prop}\label{lem1}
Let  $\O\subset\R^d~(d=2,3)$ be a bounded domain with smooth boundary  and  $p> 3/2$. Assume that the tensor-valued function   $A^{k\ell}_{ij}(x,t)\in C(\bar{\O}_T)$ satisfies   $A^{k\ell}_{ij}(\cdot, t)\in W^{1,q}(\O)$ for almost every $t\in [0,T]$, with  $\frac 1	q<\frac 1d+\min (\frac 1p,\frac{p-1}p,\frac 1d)$, and the strong Legendre condition, i.e.  there exists two positive constants $\Lambda>\lambda>0$ such that
\begin{equation}\label{eq:1.26}
\lambda |\xi|^2\leq \xi_{ki}A^{k\ell}_{ij}(x,t)\xi_{j\ell}\leq \Lambda |\xi|^2,\quad \forall  \xi\in\R^{3\times 3},\quad \forall (x,t)\in \bar{\O}\times (0,T),
\end{equation}
then for any $v^0\in W^{2-2/p,p}(\O;\R^3)$ with $\div v^0=0$ and $f\in L^p(\O_T;\R^3)$,  the system
\begin{equation*}
  \left\{
  \begin{array}{rl}
  \p_t v_k&= A^{k\ell}_{ij}(x,t) \pa_i\pa_j v_\ell+\p_kP+f_k,~\text{with}~1\leq k\leq d,\\
    \div v&=0,\\
    v|_{t=0}&=v^0,\\
    v|_{\p\O}&=0
  \end{array}
  \right.
\end{equation*}
has a unique solution $(v,P)$ such that $$v\in L^p(0,T;W^{2,p}(\O)),~\text{and}~  v_t,\nabla P\in L^p(\O_T).$$
Moreover, the following estimate holds for some constant  $C$ that is independent of $v^0$ and $f$:
\begin{equation*}
\|v_t\|_{L^p(\O_T)}+\|v\|_{L^p(0,T;W^{2,p}(\O))}+\|\nabla P\|_{L^p(\O_T)}\leq C(\|v^0\|_{W^{2-2/p,p}(\O)}+\|f\|_{L^p(\O_T)}).
\end{equation*}
The result is still valid when $p=3/2$ and $v^0\equiv 0$.
\end{prop}
 In the sequel, we shall also need the stationary version of the above result when $p=\frac 32$:
 \begin{corollary}\label{lem2}
Assume   $A^{k\ell}_{ij}(x)\in C(\bar{\O})\cap W^{1,6}(\O)$ and  there exists two positive constants $\Lambda>\lambda>0$ such that
\begin{equation}\label{eq:1.31}
\lambda |\xi|^2\leq \xi_{ki}A^{k\ell}_{ij}(x)\xi_{j\ell}\leq \Lambda |\xi|^2,\quad \forall  \xi\in\R^{3\times 3},\quad \forall x\in \bar{\O} ,
\end{equation}
then for any  $f\in L^{3/2}(\O;\R^3)$,  the system
\begin{equation}\label{eq:1.32}
  \left\{
  \begin{array}{rl}
\pa_i(  A^{k\ell}_{ij}(x) \pa_j v_\ell)+\p_kP&=f_k,~(1\leq k\leq d)\\
    \div v&=0,\\
    v|_{\p\O}&=0  \end{array}
  \right.
\end{equation}
has a unique solution $(v,P)$ with $v\in W^{2,3/2}(\O),~\nabla P\in L^{3/2}(\O)$ and the following estimate holds
\begin{equation*}
 \|v\|_{ W^{2,3/2}(\O)}+\|\nabla P\|_{L^{3/2}(\O)}\leq C(1+\|A\|_{W^{1,6}(\O)})\|f\|_{L^{3/2}(\O)}.
\end{equation*}
Moreover, if $f\in L^2(\O)$, we have the improved estimate
\begin{equation*}
 \|v\|_{ W^{2,2}(\O)}+\|\nabla P\|_{L^{2}(\O)}\leq C(1+\|A\|^2_{W^{1,6}(\O)})\|f\|_{L^{2}(\O)}.
\end{equation*}
In the above two inequalities,   $C$ only depends on the continuous modulus of $A^{k\ell}_{ij}(x)$ and geometric information of $\O$.
\end{corollary}
\begin{proof}
Since $H^1(\O)\hookrightarrow L^6(\O)$, it follows from  duality that  $L^{6/5}(\O)\hookrightarrow H^{-1}(\O)$ and this implies that $f\in L^{3/2}(\O)\hookrightarrow H^{-1}(\O)$. The assumption on $A^{k\ell}_{ij}(x)$, especially \eqref{eq:1.31}, implies that  the bilinear form
$$a(u,v):=\int_\O \p_i u_{k}A^{k\ell}_{ij}(x)\p_j v_{\ell}\ud x$$ is coercive on $H^1_{0,\sigma}(\O)$ and the existence of solution $u\in H^1_{0,\sigma}(\O)$ to \eqref{eq:1.32} follows from Lax-Milgram theorem.
In order to obtain the $L^p$-estimate of \eqref{eq:1.32}, we set $v(x,t)=\zeta(t)  u(x)$ where $\zeta$ is a non-negative smooth function such that $\zeta(t)=0$ for $t\leq 0$ and $\zeta(t)=1$ for $t\geq 1$.  It can be verified that $v(x,t)$ satisfies the following equation on $\Omega\times [0,2]$ in the sense of distribution:
\begin{equation}\label{eq:stokes4}
  \left\{
  \begin{array}{rl}
  \p_t v_k-  A^{k\ell}_{ij}(x)\pa_i\pa_j v_\ell-\p_k (\zeta(t)P)&=\zeta(t)\p_iA_{ij}^{k\ell}(x)\p_j u_\ell+\zeta'(t)u-\zeta(t)f(x),\\
    \nabla\cdot v&=0,\\
    v|_{\(\p \Omega\times [0,2]\)\cup \(\{t=0\}\times \Omega\)}&=0.
  \end{array}
  \right.
\end{equation}
It follows from $A_{ij}^{k\ell}\in W^{1,6}(\O)$ and $u\in H^1_{0,\sigma}(\O)$ that
$$\check{f}:=\zeta(t)\p_iA_{ij}^{k\ell}(x)\p_j u_\ell+\zeta'(t)u-\zeta(t)f(x)\in  L^{3/2}(\Omega_T). $$
Using Proposition \ref{lem1}
\begin{equation*}
  \|v\|_{L^{3/2}(0,2;W^{2,3/2}(\Omega))}+\|\nabla (\zeta(t) P)\|_{L^{3/2}(\O\times (0,2) )}\leq C\| \check{f} \|_{L^{3/2}(\Omega\times (0,2))},
\end{equation*}
and this leads to
\begin{equation*}
  \|u\|_{W^{2,3/2}(\Omega)}+\|\nabla P\|_{L^{3/2}(\Omega)}\leq C \(1+\|\nabla A\|_{L^{6}(\Omega)}\)\|f\|_{L^{3/2}(\O)}.
\end{equation*}
To prove the second equality in the statement, note that  $u\in W^{2,3/2}(\O)\hookrightarrow W^{1,3}(\O)$ and together with $f\in L^2(\O)$ improves the estimate for $\check{f}$
\begin{equation*}
\|\check{f}\|_{L^2(\O\times (0,2))}\leq C (1+\|\nabla A\|^2_{L^{6}(\Omega)})\|f\|_{L^2(\O)}.
\end{equation*}
So applying Proposition \ref{lem1} to  solve \eqref{eq:stokes4} again leads to the second inequality.
\end{proof}

\subsection{Abstract parabolic equation}
Following the method in \cite{AbelsDolzmannLiu2014}, we shall prove the regular in time solution with the aid of  the following result:
\begin{prop}\label{Wlokaregularity}
Suppose that $\VWloka$ and $\HWloka$ are two separable Hilbert spaces such that
the embedding $\VWloka\hookrightarrow \HWloka$
is injective, continuous, and dense.
Fix $T\in (0,\infty)$. Suppose that a bilinear
form $a( \cdot, \cdot): \VWloka\times \VWloka\mapsto \R$ is given which satisfies
for all $\phi$, $\psi\in \VWloka$ the following assumptions:
\begin{itemize}
 \item [(a)] there exists a constant $c>0$, independent of $  \phi$  and $\psi$, with
\begin{align*}
 \bbar{ a(\phi, \psi) } \leq c \| \phi\|_\VWloka\| \psi\|_\VWloka\,;
\end{align*}
\item [(b)] there exist $k_0$, $\alpha >0$ independent of   $\phi$,
with
\begin{align*}
  a( \phi, \phi) + k_0 \| \phi\|_\HWloka^2 \geq \alpha \| \phi\|_\VWloka^2 \,;
\end{align*}
\end{itemize}
Then there exists a representation operator $L : \VWloka \mapsto \VWloka^\prime$
with $a( \phi, \psi) = \scp{ L \phi}{\psi }_{\VWloka^\prime,\VWloka}$,
which  is continuous and linear. Moreover, for all
$f\in L^2((0,T);\VWloka^\prime)$ and $y_0\in \HWloka$, there exists a
unique solution
\begin{align*}
 y\in \bset{ v:[0,T]\mapsto \VWloka \text{ with }v\in L^2(0,T;\VWloka),\,
\dt{v}\in L^2(0,T;\VWloka^\prime) }
\end{align*}
solving the equation
\begin{align*}
\dt{y}+  L  y = f \quad \text{ in } \VWloka^\prime\ \text{for a.e. }t\in (0,T)\,,
\end{align*}
subject to the initial condition $y(0) =y_0$.
Finally, assume additionally that $y_0 \in \VWloka$. Then
$L: \SonetwoWloka((0,T);\VWloka) \mapsto  \SonetwoWloka((0,T);\VWloka^\prime)$
is continuous and for all $f\in \SonetwoWloka((0,T);\VWloka^\prime)$
which satisfy the compatibility condition $f(0) \in \HWloka$, 
the solution $y$ satisfies
\begin{align*}
 y\in \SonetwoWloka((0,T);\VWloka)\quad \text{ and }\quad
 \p_t^2 y\in L^2\barg{(0,T);\VWloka^\prime}\,.
\end{align*}
\end{prop}
The proof of this theorem can be found in  \cite[Lemma 26.1 and Theorem 27.2]{WlokaPDE}.

\subsection{Anisotropic Laplacian}

 We consider the following bilinear form
\begin{equation}\label{yuning:bilinear}
  a( \Psi  ,\Phi)= \int_{\Omega}\Big(L_1 \Psi_{i j,k}\Phi_{i j,k}+L_2\Psi_{ij,j}\Phi_{ik,k}
+L_3\Psi_{ij,k}\Phi_{ik,j}\Big) \ud x.
\end{equation}
for $\Psi,\Phi\in H^1(\Omega;\mathcal{Q})$.
\begin{lemma}\label{lemma1}
  For any $f\in H^{-1}(\Omega;\mathcal{Q})$, there exists  a unique $ Q \in H^1_0(\Omega;\mathcal{Q})$ such that
  \begin{equation*}
   a( Q ,\Phi)=\weight{f,\Phi}~\text{for any}~\Phi\in H^1_0(\Omega;\mathcal{Q}),
     \end{equation*}
     and there exists a constant $C$ depending on the geometry of $\O$ such that
     \begin{equation*}
     \|Q\|_{H^1_0(\Omega)}\leq C \|f\|_{H^{-1}(\Omega)}.
     \end{equation*}
\end{lemma}
This lemma can  be proved using the construction in the proof of Lemma \ref{lem:estimate-L} below. However, we present  a simpler proof here :
\begin{proof}[Proof of Lemma \ref{lemma1}]
One can verify that $a( Q ,\Phi)=\weight{\mathcal{L}( Q ),\Phi}_{H^{-1},H^1_0}$ where $\mathcal{L}$ is the operator defined by \eqref{qequ}.
   In order to apply Lax-Milgram theorem to deduce the existence of solution to \eqref{yuning:bilinear}, we need to show that $a(\cdot,\cdot)$ is coercive in $H^1_0(\Omega;\mathcal{Q})$:
  \begin{equation}\label{yuning:lowerbd1}
  \begin{split}
     a( Q , Q )\geq &\lambda\int_{\Omega} |\nabla Q |^2\ud x,
  \end{split}
\end{equation}
for some $\lambda>0$.
Note that, it suffices to prove the above inequality for smooth functions that vanishes on $\p\O$. Actually, for any $ Q \in H_0^1(\Omega;\mathcal{Q})$, we choose $ Q _n\in C_0^\infty(\overline{\Omega};\mathcal{Q})$ such that $ Q _n\to  Q $ strongly in $H^1_0(\Omega;\mathcal{Q})$. Then the conclusion follows from the continuity of $a(\cdot,\cdot)$.

Now we focus on the proof of \eqref{yuning:lowerbd1} for $Q\in C_0^\infty(\overline{\Omega})$.
It follows from \eqref{eq:1.02} and \eqref{eq:1.03} that
\begin{equation*}
\begin{split}
 |\nabla\wedge Q |^2&= |\ve^{ijk}Q_{\ell j,k}|^2=  \ve^{ijk}Q_{\ell j,k}\ve^{imn}Q_{\ell m,n}
= (\delta^{j}_{m}\delta^{k}_{n}-\delta^{j}_{n}\delta^{k}_m)Q_{\ell j,k}Q_{\ell  m,n}\\
&=  Q_{ij,k}Q_{ij,k}-Q_{ij,k}Q_{ik,j}= |\nabla  Q |^2-Q_{ij,k}Q_{ik,j}.
\end{split}
\end{equation*}
Therefore by  $Q|_{\p\O}=0$,
\begin{align*}\nonumber
    \int_{\Omega} Q_{ij,j}Q_{ik,k}\ud x =&\int_{\Omega}\(Q_{ij,k}Q_{ik,j}+ \p_j(Q_{ij}Q_{ik,k}) -\p_k(Q_{ij}Q_{ik,j})\)\ud x\\ \nonumber
    =&\int_{\Omega} \(\nabla Q |^2-|\nabla\wedge Q |^2\)\ud x
    +\int_{\partial \Omega}\(Q_{ij}Q_{ik,j}\nu_k-Q_{ij}Q_{ik,k}\nu_j\)\ud S\\
    =&\int_{\Omega} \(|\nabla Q |^2-|\nabla\wedge Q |^2\)\ud x.
\end{align*}
The above two formula together implies
\begin{align}\label{yuning:lowerbd}
     a( Q , Q )=&\int_\O\Big(L_1|\nabla Q |^2+L_2Q_{ij,j}Q_{ik,k}
+L_3Q_{ij,k}Q_{ik,j}\Big) \ud x\nonumber\\ \nonumber
=&\int_\O L_1|\nabla Q |^2\ud x+\int_\O(L_2+L_3)\(|\nabla  Q |^2-|\nabla\wedge Q |^2\)\ud x.
\end{align}
Therefore, it is easy to see that $$a(Q,Q)\ge\int_\O L_1|\nabla Q |^2 \ud x $$ when $L_2+L_3\ge 0,$
 and
 $$a(Q,Q)\ge\int_\O (L_1+L_2+L_3)|\nabla Q |^2 \ud x= L_0\int_\O |\nabla Q |^2\ud x$$ when $L_2+L_3\le 0$.
\end{proof}

The validity of the $L^p$-estimate requires the  verification of  the strong Legendre condition for
$\mathcal{L}$ in \eqref{eq:1.11}. To this end, we consider the following second order operator defined for $Q\in H^2(\O;\mathbb{R}^{3\times 3})$:
\begin{equation*}
\begin{split}
(\tilde{\mathcal{L}} (Q) )_{ij}=&~L_1\Delta Q_{ij}-\frac14(L_2+L_3)\Big(\p_i\partial_{k}Q_{jk}+\p_j\partial_{ k}Q_{ik}+\p_i\partial_{k}Q_{kj}+\p_j\partial_{ k}Q_{ki}\\
&\quad-\frac43\partial_{i}\partial_jQ_{kk}-\frac43\delta^i_j\p_k\partial_{\ell}Q_{k\ell }+\frac49\delta^i_j\Delta Q_{kk}\Big).
\end{split}
\end{equation*}

\begin{lemma}\label{lem:estimate-L}
Let $p>1$ be fixed.
For any $F\in L^p(\O; \mathbb{R}^{3\times 3})$ and $g\in W^{2-1/p,p}(\p\O; \mathbb{R}^{3\times 3})$, there exists a unique $ Q \in W^{2,p}(\O; \mathbb{R}^{3\times 3})$ that solves $\tilde{\mathcal{L}} Q =F$ with boundary condition $Q|_{\p\O}=g$. Moreover, there exists $C>0$ depending only on $\O$ such that
\begin{equation*}
  \| Q \|_{W^{2,p}(\Omega)}\leq C\(\|g\|_{W^{2-1/p,p}(\p\Omega)}+\|F\|_{L^p(\Omega)}\).
\end{equation*}
Especially, when $F,g\in \mathcal{Q}$, we have $Q\in W^{2,p}(\Omega;\mathcal{Q})$ satisfying $\mathcal{L}Q=F$.
\end{lemma}
\begin{proof}
It suffices to verify the strong Legendre condition (see \eqref{eq:1.26} for instance) for $\tilde{\mathcal{L}}$ and then the conclusion follows from standard theory of elliptic system (cf. \cite[Chapter IV]{AgmonDouglisNirenberg1964} or \cite[Chapter 10]{ChenWu1998} ).
To this end, we first note that  $\tilde{\mathcal{L}}$ can be written as
\begin{equation*}
\begin{split}
(\tilde{\mathcal{L}}(Q) )_{ij}=&-L_1\partial_\ell\Big(\delta^\ell_k\delta^i_{i'}\delta^j_{j'}\partial_kQ_{i'j'}\Big)
-\frac14(L_2+L_3)\partial_\ell\Big((\delta^j_{i'}\delta^k_{i}\delta^\ell_{j'}+\delta^i_{i'}\delta^\ell_{j'}\delta^k_j+\delta_k^i\delta_\ell^{i'}\delta_j^{j'}+\delta_k^j\delta_\ell^{i'}\delta_i^{j'})\partial_kQ_{i'j'}\\
&-\frac43\delta^{i'}_{j'}\delta^{i}_\ell\delta^{j}_k\partial_kQ_{i'j'}-\frac43\delta^i_j\delta^{i'}_\ell\delta^{j'}_k\partial_kQ_{i'j'}
+\frac49\delta^i_j\delta^{i'}_{j'}\delta^k_\ell\partial_kQ_{i'j'}\Big)\\
=&-\partial_\ell\Big(A^{\ell k}_{(ij)(i'j')}\partial_kQ_{i'j'}\Big),
\end{split}
\end{equation*}
where
\begin{equation}\label{eq:1.09}
\begin{split}
A^{\ell k}_{(ij)(i'j')}=&~L_1\delta^\ell_k\delta^i_{i'}\delta^j_{j'}+\frac14(L_2+L_3)
\big(\delta^j_{i'}\delta^k_{i}\delta^\ell_{j'}+\delta^i_{i'}\delta^\ell_{j'}\delta^k_j+\delta_k^i\delta_\ell^{i'}\delta_j^{j'}+\delta_k^j\delta_\ell^{i'}\delta_i^{j'}\\
&-\frac43\delta^{i'}_{j'}\delta^{i}_\ell\delta^{j}_k-\frac43\delta^i_j\delta^{i'}_\ell\delta^{j'}_k+\frac49\delta^i_j\delta^{i'}_{j'}\delta^k_\ell\big).
  \end{split}
\end{equation}
To verify the strong Legendre condition for $A^{\ell k}_{(ij)(i'j')}$, we need to compute
\begin{equation*}
  \begin{split}
    A^{\ell k}_{(ij)(i'j')}\xi_\ell^{(ij)}\xi_k^{(i'j')}
=&L_1\sum_{i,j,\ell} (\xi_\ell^{(ij)})^2+\frac14(L_2+L_3)\Big(\xi^{(ij)}_{\ell}\xi^{(j\ell)}_i+\xi^{(ij)}_{\ell}\xi^{(i\ell)}_j+\xi^{(ij)}_{\ell}\xi^{(\ell j)}_i+\xi^{(ij)}_{\ell}\xi^{(\ell i)}_j\\
&-\frac43\xi^{(ij)}_{i}\xi^{(\ell \ell)}_{j}-\frac43\xi^{(ii)}_j\xi^{(j\ell )}_\ell+\frac49\sum_\ell\big|\xi^{(ii)}_\ell\big|^2\Big).
  \end{split}
\end{equation*}
To this end,  we define a new tensor by $\zeta^{(ij)}_{\ell}=\xi^{(ij)}_{\ell}-\frac13\delta^i_j\xi^{(kk)}_\ell$. Then it is easy to verify  that
\begin{align*}
 \sum_{i,j,\ell} |\zeta^{(ij)}_\ell|^2 \le \sum_{i,j,\ell}|\xi^{(ij)}_\ell|^2.
\end{align*}
Thus we have, for the case $L_2+L_3\le 0$,
\begin{equation*}
  \begin{split}
    A^{\ell k}_{(ij)(i'j')}\xi_\ell^{(ij)}\xi_k^{(i'j')}
=&L_1 \sum_{i,j,\ell}\big|\xi_\ell ^{(ij)}\big|^2+\frac14(L_2+L_3)\Big(\zeta^{(ij)}_{\ell }\zeta^{(j\ell )}_i+\zeta^{(ij)}_{\ell}\zeta^{(i\ell)}_j+\zeta^{(ij)}_{\ell}\zeta^{(\ell j)}_i+\zeta^{(ij)}_{\ell}\zeta^{(\ell i)}_j\Big)\\
\geq &L_1 \sum_{i,j,\ell}\big|\xi_\ell^{(ij)}\big|^2+\frac14\sum_{i,j,\ell}(L_2+L_3)\Big(\big|\zeta^{(ij)}_{\ell}\big|^2+\big|\zeta^{(i\ell)}_j\big|^2+\big|\zeta^{(ij)}_{\ell}\big|^2+\big|\zeta^{(j\ell)}_i\big|^2\Big)\\
\geq &L_1 \sum_{i,j,\ell}\big|\xi_\ell^{(ij)}\big|^2+\frac14\sum_{i,j,\ell}(L_2+L_3)\Big(\big|\xi^{(ij)}_{\ell}\big|^2+\big|\xi^{(i\ell)}_j\big|^2+\big|\xi^{(ij)}_{\ell}\big|^2+\big|\xi^{(j\ell)}_i\big|^2\Big)\\
=&(L_1+L_2+L_3)\sum_{i,j,\ell}\big|\xi_\ell^{(ij)}\big|^2.
  \end{split}
\end{equation*}
In the case when $L_2+L_3\ge 0$, we can write
\begin{equation*}
  (\tilde{\mathcal{L}} (Q) )_{ij}=-\partial_\ell\Big(\tilde{A}^{\ell k}_{(ij)(i'j')}\partial_kQ_{i'j'}\Big),
\end{equation*}
where $ \tilde{A}^{\ell k}_{(ij)(i'j')}$ is defined by  interchanging $k,\ell$ in \eqref{eq:1.09}:
\begin{equation*}
\begin{split}
\tilde{A}^{\ell k}_{(ij)(i'j')}=&~L_1\delta^\ell_k\delta^i_{i'}\delta^j_{j'}+\frac14(L_2+L_3)
\big(\delta^j_{i'}\delta^\ell_{i}\delta^k_{j'}+\delta^i_{i'}\delta^k_{j'}\delta^\ell_j+\delta_\ell^i\delta_k^{i'}\delta_j^{j'}+\delta_\ell^j\delta_k^{i'}\delta_i^{j'}\\
&-\frac43\delta^{i'}_{j'}\delta^{i}_k\delta^{j}_\ell-\frac43\delta^i_j\delta^{i'}_k\delta^{j'}_\ell+\frac49\delta^i_j\delta^{i'}_{j'}\delta^k_\ell\big).
  \end{split}
\end{equation*}
This yields
\begin{equation*}
  \begin{split}
    A^{\ell k}_{(ij)(i'j')}\xi_\ell^{(ij)}\xi_k^{(i'j')}
=&L_1 \sum_{i,j,\ell}\big|\xi_\ell^{(ij)}\big|^2+\frac14(L_2+L_3)\Big(\xi^{(ij)}_{i}\xi^{(j\ell)}_\ell+\xi^{(ij)}_{j}\xi^{(i\ell)}_\ell+\xi^{(ij)}_{i}\xi^{(\ell j)}_\ell+\xi^{(ij)}_{j}\xi^{(\ell i)}_\ell\\
&-\frac43\xi^{(ij)}_{j}\xi^{(\ell\ell)}_{i}-\frac43\xi^{(ii)}_{\ell}\xi^{(j\ell)}_j+\frac49\sum_\ell\big|\xi^{(ii)}_\ell\big|^2\Big)\\
=&L_1 \sum_{i,j,\ell}\big|\xi_\ell^{(ij)}\big|^2+\frac14(L_2+L_3)\sum_{j}\Big|\xi^{(ij)}_{i}+\xi^{(ji)}_i-\frac23\xi^{(ii)}_{j}\Big|^2
\ge  L_1\sum_{i,j,\ell} \big|\xi_\ell^{(ij)}\big|^2
  \end{split}
\end{equation*}
and
we   conclude that $\tilde{\mathcal{L}}$ satisfies the strong Legendre condition in both cases.

To prove the `especially' part, we first note that    $\tilde{\mathcal{L}} (Q)={\mathcal{L}} (Q)$ for $Q\in\mathcal{Q}$.
Consequently, if $Q\in\mathcal{Q}$ solves $\mathcal{L}(Q)= F $ for some $ F $ with image in $\mathcal{Q}$, then  $\tilde{\mathcal{L}}(Q)= F $. On the other hand, if $ F \in\mathcal{Q}$, and
$Q\in\mathbb{R}^{3\times 3}$ solves $\tilde{\mathcal{L}}(Q)= F $, then we have
\begin{equation*}
L_1\Delta(Q-Q^T)=  \tilde{\mathcal{L}}(Q)-(\tilde{\mathcal{L}}(Q))^T=0,
\end{equation*}
and
\begin{equation*}
L_1\Delta(\tr Q)=\tr  \tilde{\mathcal{L}}(Q) =\tr F =0.
\end{equation*}
Thus $Q\in\mathcal{Q}$ and $\mathcal{L}(Q)= F $.
\end{proof}
The above lemma implies the following:
\begin{corollary}\label{lem:regularity-L}
The operator $\mathcal{L}:H^2(\O;\mathcal{Q})\cap H^1_0(\O;\mathcal{Q})\mapsto L^2(\O;\mathcal{Q}))$ defined by \eqref{qequ} is an isomorphism.
\end{corollary}

\section{Abstract form of the system}\label{fasetup}
The task of this section is to setup the functional analytic framework for \eqref{yuning:be} and \eqref{yuning:BE-ini-bou}.
We first remark that the Beris-Edward system \eqref{yuning:be} obeys the basic energy dissipation law
\begin{equation}\label{eq:energy-law}
  \frac{d}{dt}\(\mathcal{F}( Q,\nabla Q )+\int_{\Omega} \frac 12| u|^2\ud x\)+\int_{\Omega}\( |\nabla u|^2+\left| \mathcal{H}(Q)\right|^2\)\ud x=0.
\end{equation}
This can be formally   done by  first testing  equation \eqref{eq:BE-v} by the velocity field $u$ and  testing \eqref{eq:BE-Q} by $ \mathcal{H}(Q)$ in \eqref{eq:1.11}, then simple  integration by parts lead to:
\begin{equation}\label{final1}
\begin{split}
&\frac{d}{dt}\int_{\Omega} \frac 12|u|^2\ud x+\int_{\Omega} |\nabla u|^2\ud x\\
=&-\int_{\Omega} \(-\mathcal{S}_Q ( \mathcal{H}  )
+ Q \cdot \mathcal{H}  - \mathcal{H}  \cdot Q \):\nabla u \ud x-\int_{\Omega} \partial_j\(\frac{\partial \mathcal{F}}{\partial Q_{k\ell,j}}Q_{k\ell,i}\)u_i\ud x,
\end{split}
\end{equation}
and
\begin{equation}\label{final2}
  \begin{split}
  &-\frac{d}{dt}\mathcal{F}( Q,\nabla Q )+\int_{\Omega}  u\cdot\nabla Q : \mathcal{H} (Q) \ud x\\
  =&\int_{\Omega}\(| \mathcal{H} (Q) |^2+\big(\mathcal{S}_Q (D(u))+W(u)\cdot  Q - Q \cdot W(u)\big): \mathcal{H}\)\ud x .
  \end{split}
\end{equation}
Since we have
\begin{align*}
\int_{\Omega} \partial_j\(\frac{\partial \mathcal{F}}{\partial Q_{k\ell,j}}Q_{k\ell,i}\)u_i\ud x
=&\int_{\Omega} \(\partial_j\(\frac{\partial \mathcal{F}}{\partial Q_{k\ell,j}}\)Q_{k\ell,i}+\frac{\partial \mathcal{F}}{\partial Q_{k\ell,j}}Q_{k\ell,ij}\)u_i\ud x\\
=&\int_{\Omega} \(\mathcal{H}_{k\ell}(Q)Q_{k\ell,i}+\frac{\partial \mathcal{F}}{\partial Q_{k\ell}}Q_{k\ell,i}+\frac{\partial \mathcal{F}}{\partial Q_{k\ell,j}}Q_{k\ell,ij}\)u_i\ud x\\
=&\int_{\Omega} \(\mathcal{H}_{k\ell}(Q)Q_{k\ell,i}+\partial_i \mathcal{F}(Q,\nabla Q)\)u_i\ud x\\
=&\int_{\Omega} \mathcal{H}_{k\ell}(Q)Q_{k\ell,i}u_i\ud x,
\end{align*}
thus subtracting \eqref{final1} from \eqref{final2}  and using the following cancellation yields (\ref{eq:energy-law}).

The following lemma indicates the important cancellation law between nonlinear terms:
\begin{lemma}
For any $Q\in\mathcal{Q}$, the linear operator $\mathcal{S}_Q ( M)$ defined by \eqref{eq:1.15} is    symmetric  and traceless.
Moreover, for any $3\times 3$ matrices $P$ and $M$, it holds
\begin{equation}\label{yuning:cal1'}
\mathcal{S}_Q( M):P  =\mathcal{S}_Q(P):M.
\end{equation}
In addition, if $M$ is symmetric, then
\begin{equation}\label{yuning:cal1}
\( -\mathcal{S}_Q( M)+ Q \cdot M- M\cdot Q \):P
  =\(-\mathcal{S}_Q(S)+ Q \cdot A-A\cdot Q \): M,
\end{equation}
where $S$ and $A$ are symmetric and anti-symmetric parts of $P$ respectively.
\end{lemma}
\begin{proof}
Note that the space consisting of all $k$ by $k$ matrices under the Frobenius product $A:B= \tr AB^T =  A_{ij}B_{ij}$ is a Hilbert space and it allows the following direct product decomposition:
\begin{equation*}
\mathbb{R}^{n\times n}=\{M\in \mathbb{R}^{n\times n}, M_{ij}=M_{ji}\}\oplus \{M\in \mathbb{R}^{n\times n}, M_{ij}=-M_{ji}\}.
\end{equation*}
Actually, any   $M$ can be uniquely written as the sum of two orthogonal parts $$M=\frac{M+M^T}2+\frac{M-M^T}2.$$
Then direct calculations implies the  identity  \eqref{yuning:cal1'}. Formula \eqref{yuning:cal1} is trickier and can be proved using \eqref{yuning:cal1'}:
  \begin{equation*}
\begin{split}
   &( -\mathcal{S}_Q( M)+ Q \cdot M- M^T\cdot Q ): P\\
  =&  ( -\mathcal{S}_Q( M)+ Q \cdot M- M^T\cdot Q ):(S+A)\\
  =&-\mathcal{S}_Q( M):S+( Q \cdot M- M^T\cdot Q ):A\\
  =&-\mathcal{S}_Q(S): M+( Q \cdot A-A \cdot Q ): M.
\end{split}
\end{equation*}

\end{proof}

The typical situation for the application of \eqref{yuning:cal1} is when $P=\nabla u=D(u)+W(u)$ for some   vector field $u$:
\begin{equation*}
\Big( -\mathcal{S}_Q( M)+ Q \cdot M- M\cdot Q \Big):\nabla u
  =\Big(-\mathcal{S}_Q(D(u))+ Q \cdot W(u)-W(u)\cdot Q \Big): M,\quad \forall Q,M\in\mathcal{Q}.
\end{equation*}

Let  $(u_0,Q_0)\in \Zspace$, defined by \eqref{init}. As usual,  the first equation in \eqref{yuning:be} will be formulated   by testing with divergence-free vector fields, or equivalently, by applying the Leray's projector
\begin{equation} \label{pressure}
\left\{\begin{array}{rl}
  &u_t=P_\sigma \div \( -u\otimes u + \sigma^{s}(u,Q)+\sigma^{a}(Q)+\sigma^d(Q) \),\quad \\
&Q_t+u\cdot\nabla  Q  + Q \cdot W(u)-W(u)\cdot Q =  \mathcal{H}(Q)  +\mathcal{S}_Q (D(u)),
\end{array}
\right.
\end{equation}
where $P_\sigma\colon H^{-1}( \Omega ;\R^3)\mapsto H^{-1}_\sigma( \Omega )$ and $\div\colon L^2( \Omega ;\R^{3\times 3})\mapsto H^{-1}( \Omega ;\R^3)$ are defined  in Section~\ref{sec:Notation}.
The idea is to rewrite the nonlinear system~\bref{pressure} as an
abstract evolutionary  equation in a suitable Banach space. With the notation introduced in Section \ref{sec:Notation},
we define the linearized operator at the initial director field $ Q_0 $ by
\begin{equation}\label{linop}
\mathscr{L}_{ Q_0 }\begin{pmatrix} u\\ Q \end{pmatrix}:=
\frac{d}{dt}\begin{pmatrix} u\\ Q \end{pmatrix}
-\begin{pmatrix}
  P_{\sigma}\div\left[ D( u)+\mathcal{S}_{ Q _0}(\mathcal{L}( Q ))
- Q _0\cdot \mathcal{L}( Q )+\mathcal{L}( Q )\cdot Q _0\right]\\
-\mathcal{L}( Q )+\mathcal{S}_{ Q _0}( D( u))- Q _0\cdot W(u)+W(u)\cdot Q _0
\end{pmatrix}\,
\end{equation}
and nonlinear part is given by
 \begin{equation}\label{yuning:nonlinear1}
  \begin{split}
   \tilde{ \mathscr{N}}_{ Q _0}\begin{pmatrix}  u\\  Q \end{pmatrix}&:=
\begin{pmatrix}
  P_{\sigma}\div\left[ \mathcal{S}_Q(\mathcal{J}( Q ))
- Q \cdot\mathcal{J}( Q )+\mathcal{J}( Q )\cdot Q - u\otimes u-\frac{\partial\mathcal{F}}
{\partial\nabla Q }:\nabla Q \right]\\
-\mathcal{J}( Q )- u\cdot\nabla Q
\end{pmatrix}\\
&+
\begin{pmatrix}
  P_{\sigma}\div\left[\mathcal{S}_Q(\mathcal{L}( Q ))-\mathcal{S}_{ Q _0}(\mathcal{L}( Q ))
  -( Q - Q _0)\cdot\mathcal{L}( Q )+\mathcal{L}( Q )\cdot( Q - Q _0)\right]\\
\mathcal{S}_Q( D( u))-\mathcal{S}_{ Q _0}(D( u))+( Q _0- Q )\cdot W(u)
-W(u)\cdot( Q _0- Q )
\end{pmatrix}.
  \end{split}
\end{equation}
So if $(u,Q)$ is a solution to    \eqref{yuning:be} satisfying  initial-boundary conditions \eqref{yuning:BE-ini-bou}, then $(u_h,Q_h)=(u-u_0,Q-Q_0)$ satisfies the  operator equation
\begin{equation}\label{oper3}
  \mathscr{L}_{Q_0}\begin{pmatrix} \uh+u_0\\ \Qh+Q_0 \end{pmatrix}
=\tilde{\mathscr{N}}_{Q_0}\begin{pmatrix} \uh+u_0\\ \Qh+Q_0 \end{pmatrix}\,,
 \end{equation}
as well as the  homogeneous initial-boundary conditions. 	

Due to the inhomogenous boundary conditions, the operator $\tilde{\mathscr{N}}_{Q_0}$    is defined on an affine space.
For the purpose  of applying classical result in functional analysis and operator theory, we shall rewrite it as an nonlinear operator between two Banach spaces. To this end, we   denote the stationary version of \eqref{linop} by
\begin{equation*}
\mathscr{S}_{ Q_0 }\begin{pmatrix} u\\ Q \end{pmatrix}:=\begin{pmatrix}
  P_{\sigma}\div\left[ D( u)+\mathcal{S}_{ Q _0}(\mathcal{L}( Q ))
- Q _0\cdot \mathcal{L}( Q )+\mathcal{L}( Q )\cdot Q _0\right]\\
-\mathcal{L}( Q )+\mathcal{S}_{ Q _0}( D( u))- Q _0\cdot W(u)+W(u)\cdot Q _0
\end{pmatrix}\,.
\end{equation*}
Then it follows from the linearity of \eqref{linop} and the assumption that $(u_0,Q_0)$ is time-independent that, equation \eqref{oper3} is   equivalent to
\begin{equation}\label{trans}
  \mathscr{L}_{Q_0}
\begin{pmatrix} \uh\\ \Qh \end{pmatrix}=
\tilde{\mathscr{N}}_{Q_0}\begin{pmatrix} \uh+u_0\\ \Qh+Q_0 \end{pmatrix}+
\mathscr{S}_{Q_0}\begin{pmatrix} u_0\\ Q_0 \end{pmatrix}.
\end{equation}
The right-hand side of \eqref{trans} is a translated version of \eqref{yuning:nonlinear1} and is a mapping between   linear spaces rather than  affine spaces:
\begin{equation} \label{newnon}
   \mathscr{N }_{(u_0,Q_0)}\begin{pmatrix}
\uh\\
\Qh
\end{pmatrix}:=\tilde{\mathscr{N}}_{Q_0}\begin{pmatrix}
\uh+u_0\\
\Qh+Q_0
\end{pmatrix}+ \mathscr{S}_{Q_0}\begin{pmatrix}
u_0\\
Q_0
\end{pmatrix}.
\end{equation}
So we end up with the following abstract parabolic system  that is equivalent to \eqref{oper3}:
\begin{equation}\label{eq:1.06}
  \mathscr{L}_{Q_0}\begin{pmatrix} \uh \\ \Qh  \end{pmatrix}
=\mathscr{N }_{(u_0,Q_0)}\begin{pmatrix}
\uh\\
\Qh
\end{pmatrix}\,.
\end{equation}

To incorporate the initial-boundary condition \eqref{yuning:BE-ini-bou},  we  turn to the definition of functional spaces $X_0$ and $Y_0$ such that
$\mathscr{L}_{Q_0},\,\mathscr{N }_{(u_0,Q_0)}:X_0\mapsto Y_0$  with  $\mathscr{L}_{Q_0}$ being  an isomorphism. Motivated by the idea to construct solutions which are
twice differentiable in time and
the precise assertions in Theorem~\ref{thm:localstrong},
  we need to prove the existence of regular solutions of the linear equation
\begin{equation*}
  \mathscr{L}_{Q_0}\begin{pmatrix}
  \uh\\
  \Qh
\end{pmatrix} =\begin{pmatrix}
   f\\
   g
 \end{pmatrix}
\end{equation*}
subject to
homogeneous initial data
 with right-hand side $(f,g)\in Y_0$. The general linear theory requires a compatibility condition
which is taken care of by the definition of $Y_0$ as
\begin{align} \label{Y0}
  Y_0&=\left\{ (f,g)\in H^1(0,T;\VNSprimeO)\times H^1(0,T;L^2(\Omega;\mathcal{Q})):
(  f,  g)|_{t=0}\in \HNSO
\times H^1_0(\Omega) \right\}\,.
\end{align}
These spaces are equipped with the usual norms in product spaces and for spaces
of functions of one variable with values in a Banach space together with the
correct norm of the initial data.
More precisely, the norm of $Y_0$ is given by
\begin{equation*}
  \|(f,g)\|^2_{Y_0}
= \|f\|_{H^1(0,T;\VNSprimeO)   }^2+\| g\|^2_{H^1(0,T;L^2(\Omega))}+\|(f,g)|_{t=0}\|_{\HNSO\times H^1(\Omega)}^2 \,.
\end{equation*}
Note that the last part of the norm is not
controlled by applying trace theorems  to the first two parts.
Now we turn to the domain of $\mathscr{L}_{Q_0}$:
\begin{equation*}
  \begin{split}
    X_u&=H^2(0,T;\VNSprimeO) \cap H^1(0,T;\VNSO),\\
     X_Q&=H^2(0,T;L^2(\Omega;\mathcal{Q}))\cap H^1(0,T;H^2(\Omega;\mathcal{Q})),
  \end{split}
\end{equation*}
together with the norms
\begin{equation*}
\begin{split}
 \|u\|^2_{X_u}&= \|u\|_{H^2(0,T;\VNSprimeO) }^2+\|u\|_{H^1(0,T;\VNSO)}^2+\|u|_{t=0}
  \|_{\VNSO}^2+\|u_t|_{t=0}
  \|_{L^2(\Omega)}^2 \,,
\\ \label{XQ}
  \|Q\|_{X_Q}^2&= \|Q\|_{H^2(0,T;L^2(\Omega;\mathcal{Q}))}^2
  +\|Q\|_{H^1(0,T;H^2(\Omega;\mathcal{Q}))}^2+\|Q|_{t=0}\|_{ H^2(\Omega)}^2+\|\dt{Q}|_{t=0}\|_{ H^1(\Omega)}^2 \,.
\end{split}
\end{equation*}
Note that the last two terms in the norms are important to obtain in the sequel constants that are uniformly bounded as $T\to 0$.
The corresponding subspaces related to the homogeneous initial and boundary conditions in
the formulation of the problem are defined by
\begin{equation*}
X_0= \left\{(u,Q)\in X_u\times X_Q\mid Q|_{\p \Omega}=0,\left(u,Q \right)|_{t=0}
=\left(0,0 \right)\right\},
\end{equation*}
which is equipped with the product norm
\begin{equation*}
  \|(u,Q)\|_{X_0}=\|(u,Q)\|_{X_u\times X_Q}\,.
\end{equation*}

  \begin{prop}\label{lem5}
   If $(\uh,\Qh)$ is a strong solution to  \eqref{eq:1.06}, then $(u,Q)= (\uh,\Qh)+(u_0,Q_0)$ is a solution to \eqref{yuning:be}   and \eqref{eq:BE-ini}.
    \end{prop}

\section{Proof of Theorem \ref{thm:localstrong}}\label{section3}

The first step towards  the proof of the local in time existence of strong solutions  is to construct a regular in time solution, following the method in  \cite{AbelsDolzmannLiu2014}.
The following result establishes the invertibility of the linear operator equation.
Note that we are seeking a solution of the linear equation in $X_0$, i.e., a
solution with homogeneous initial and boundary conditions.
\begin{prop}
\label{linearizedprop}
For any fixed $T\in (0,1]$, $\mathscr{L}_{Q_0}$ defined by \eqref{linop} is a   bounded linear operator $X_0\mapsto Y_0$  and
for every $(f,g)\in Y_0$, the operator equation
\begin{equation}\label{eq:1.10}
   \mathscr{L}_{Q_0} (u,Q) = (f,g)
\end{equation}
has a unique solution $(u,Q)\in X_0$ satisfying
\begin{equation}\label{stronginverse}
\| \mathscr{L}^{-1}_{Q_0}(f,g)\|_{X_0}=
\|(u,Q)\|_{X_0}\leq C_{\mathscr{L}}\|(f,g)\|_{Y_0},
\end{equation}
where $C_{\mathscr{L}}$ is independent of $T\in(0,1]$.
In particular $\mathscr{L}_{Q_0}:X_0\mapsto Y_0$ is invertible and
$\mathscr{L}^{-1}_{Q_0}$ is a bounded linear operator with norm independent of $T\in(0,1]$.
\end{prop}

\begin{proof}
 In order to apply Proposition \ref{Wlokaregularity}, we define the Hilbert spaces
\begin{align*}
 \HWloka  = \HWloka_1 \times \HWloka_2
& =\HNSO\times H^1_{0}( \Omega ;\mathcal{Q}) \,,\\
\VWloka   = \VWloka_1 \times \VWloka_2
& = \VNSO \times
 \( H^2( \Omega ;\mathcal{Q})\cap H^1_{0}( \Omega ;\mathcal{Q})  \),
\end{align*}
and equip them with standard product Sobolev norm. The dual spaces of
$\VWloka$ with respect to pivot space $\HWloka$ is $$\VWloka'=\VWloka_1'\times\VWloka_2'=H^{-1}_\sigma(\Omega)\times L^2(\Omega;\mathcal{Q}),$$
and the dual product is given  by
\begin{equation}\label{eq:1.12}
  \bscp{(u,Q)}{(\varphi,\Phi)}_{\VWloka',\VWloka}:=\bscp{u}{\varphi}_{\VWloka_1' \times\VWloka_1 }+\bscp{Q}{\Phi}_{\VWloka_2' \times\VWloka_2 }
  =\bscp{u}{\varphi}_{H^{-1}_\sigma,H^1_\sigma}+\int_{\Omega} Q:\mathcal{L}(\Phi)\dv{x},
\end{equation}
according to Lemma \ref{lem:regularity-L}. As a result, the space $Y_0$ defined by \eqref{Y0}  can be written by
\begin{equation*}
  Y_0=\{(f,g) \in H^1(0,T;\VWloka')\mid (f,g)|_{t=0}\in \HWloka\}.
\end{equation*}
We shall define  the bilinear form $a(\cdot,\cdot)$ on $\VWloka$ by
\begin{align*}
a((u,Q),(v,P))=& \int_{\Omega} \Big(D(u)+\mathcal{S}_{Q_0}(\mathcal{L}(Q))-Q_0\cdot\mathcal{L}(Q)+\mathcal{L}(Q) \cdot Q_0\Big):D v \dv{x}\\
&+\int_{\Omega} \Big(\mathcal{L}(P)-\mathcal{S}_{Q_0}(D(u))+Q_0 \cdot W(u)-W(u) \cdot Q_0\Big):\mathcal{L}(P)\dv{x}.
\end{align*}
 One can verify that this bilinear form satisfies the hypothesis  for applying Proposition \ref{Wlokaregularity}. Especially, coerciveness follows from cancellation law \eqref{yuning:cal1},
  \begin{align*}
a((u,Q),(u,Q))
=& \int_{\Omega} \Big(D(u)+\mathcal{S}_{Q_0}(\mathcal{L}(Q))-Q_0\cdot\mathcal{L}(Q)+\mathcal{L}(Q)\cdot Q_0\Big):D u \dv{x}\\
&+\int_{\Omega} \Big(\mathcal{L}(Q)-\mathcal{S}_{Q_0}(D(u))+Q_0 \cdot W(u)-W(u)\cdot Q_0\Big):\mathcal{L}(Q)\dv{x}\\
=&\int_{\Omega} (|D(u)|^2+|\mathcal{L}(Q)|^2)\dv{x}
\geq C\|(u,Q)\|^2_{\VWloka }.
\end{align*}
In the last step, we employed  Corollary \ref{lem:regularity-L}.
So  there exists a bounded linear operator $L:\VWloka\mapsto \VWloka'$ such that
$$\bscp{ L (u, Q)}{(\varphi,\Phi)}_{\VWloka',\VWloka}=a((u,Q),(\varphi,\Phi)).
$$
Moreover, for any $(f,g)\in Y_0$,
 the   abstract evolution equation
  \begin{align}\label{eq:AbstractEq}
 \bscp{(\dt{u},\dt{Q})}{(\varphi,\Phi)}_{\VWloka',\VWloka}
+\bscp{ L (u, Q)}{(\varphi,\Phi)}_{\VWloka',\VWloka}
=  \scp{(f,g)}{(\varphi, \Phi)}_{\VWloka',\VWloka}
\end{align}
 has a unique solution $(u,Q)$ satisfying
\begin{align*}
 (u,Q)\in \SonetwoWloka((0,T);\VWloka)\quad \text{ and }\quad
 (\p^2_t u,\p_t^2 Q)\in L^2\barg{(0,T);\VWloka^\prime}\,.
\end{align*}
or shortly $(u,Q)\in X_0$. Now we need to show that \eqref{eq:AbstractEq} is equivalent to \eqref{eq:1.10}: it is evident that,
 choosing $(\varphi,0) \in \VWloka$ in \eqref{eq:AbstractEq} implies the first equation in \eqref{eq:1.10}. To identify the equation for $Q$, we choose $(0,\Phi) \in \VWloka$ as
test function and deduce from \eqref{eq:1.12} that
\begin{align*}
&\int_{\Omega}  g(t,x):\mathcal{L}( \Phi(x))\dv{x}\\
 &=
\weight{\dt{Q},\Phi}_{\VWloka_2',\VWloka_2}
+\int_{\Omega} \Big(\mathcal{L}(Q)-\mathcal{S}_{Q_0}(D(u))+Q_0\cdot W(u)-W(u)\cdot  Q_0\Big):\mathcal{L}(\Phi)\dv{x}
\\ &=  \int_{\Omega}  \(\dt{Q} +\mathcal{L}(Q)-\mathcal{S}_{Q_0}(D(u))+Q_0 \cdot W(u)-W(u) \cdot Q_0\):\mathcal{L}(\Phi) \dv{x}\,.
\end{align*}
In view of Lemma \ref{lem:regularity-L},   $\mathcal{L}\colon \VWloka_2\mapsto L^2( \Omega ;\mathcal{Q})$ is bijective and thus
\begin{equation*}
   \dt{Q} +\mathcal{L}(Q)-\mathcal{S}_{Q_0}(D(u))+Q_0 \cdot W(u)-W(u) \cdot Q_0 = g, \qquad \text{a.e. in }  \Omega \times (0,T).
\end{equation*}

Altogether, we have proven that $\mathscr{L}_{Q_0}\colon X_0\mapsto Y_0$
is an isomorphism. Since $\mathscr{L}_{Q_0}$ is also a bounded linear operator and the operator norm only depends   on $Q_0$ and geometry of $\O$,  the boundedness of its inverse operator
$\mathscr{L}_{Q_0}^{-1}\colon Y_0\mapsto X_0$   follows from inverse mapping theorem. The assertion that   $C_{\mathscr{L}}$ is independent of $T$  follows from standard energy estimate  and the cancellation law \eqref{yuning:cal1}. Here we omit the details.
\end{proof}

\begin{prop}\label{nonlinearterms}
Fix $0<T\leq 1$, $R>0$, $(u_0, Q_0)\in \Zspace$.
Let $\mathscr{N}_{(u_0,Q_0)}$ be the nonlinear operator defined
in \eqref{newnon} and
$\BXzeroR =\{(v,P)\in X_0,\, \|(v,P)\|_{X_0}\leq R\}$.
Then the following assertions hold
for all $(u_i,Q_i)\in \BXzeroR$, $i=1,\,2$:

\begin{itemize}
 \item [(i)] $\mathscr{N}_{(u_0,Q_0)}$ maps $X_0$ to $Y_0$.

\item [(ii)] Local Lipschitz continuity: there exists a constant
$C_{\mathscr{N}}(T,R,Q_0,u_0)>0$ such that
\begin{align}\label{lip}
\begin{aligned}
&\|\mathscr{N}_{(u_0,Q_0)}(u_1,Q_1)-\mathscr{N}_{(u_0,Q_0)}(u_2,Q_2)\|_{Y_0} \\
&\,\qquad\qquad
\leq  C_{\mathscr{N}}(T,R,Q_0,u_0)\|(u_1-u_2,Q_1-Q_2)\|_{X_0}\,.
\end{aligned}
\end{align}

\item [(iii)] Local boundedness:
\begin{equation}\label{bounded}
\|\mathscr{N}_{(u_0,Q_0)}(u_1,Q_1)\|_{Y_0}
\leq C_{\mathscr{N}}(T,R,Q_0,u_0)\|(u_1,Q_1)\|_{X_0}+
\| \mathcal{E}(u_0, Q_0)\|_{Y_0}\,.
\end{equation}
where $\mathcal{E}$ is given by \eqref{stanon}.

\item [(iv)]
For any fixed $R>0$, $\lim_{T\to 0}C_{\mathscr{N}}(T,R,Q_0,u_0)= 0$.
\end{itemize}
\end{prop}
The proof of this result can be adapted line by line  from \cite[Proposition 4.3]{AbelsDolzmannLiu2014}. Actually, the proof in \cite{AbelsDolzmannLiu2014} is slightly more general since they work with variable viscosity in the fluid equation and   mixed boundary condition for the $Q$-tensor field.

\begin{proof} [Proof of Theorem~\ref{thm:localstrong}]
The proof will be divided into two steps. First, we shall use Proposition \ref{linearizedprop} and \ref{nonlinearterms}
to prove the existence and uniqueness of a regular in time solution.  Based on this, in the second step, owning to a   special  structure
of the system, we improve the spatial regularity of $u$ and also $Q$ and this leads to the strong solution of  \eqref{yuning:be}.
\medskip

\textit{\underline{Step 1: Regularity in time.}}

We first show that
\begin{equation*}
  \mathscr{A} := \mathscr{L}^{-1}_{Q_0}\mathscr{N}_{(u_0,Q_0)}:X_0\mapsto X_0
\end{equation*}
has a unique fixed-point. By~\bref{stronginverse} and \eqref{lip} we find for all
$(\uhi,\Qhi)\in \BXzeroR$ that
\begin{align*}
&\bnorm{ \mathscr{L}_{Q_0}^{-1}\mathscr{N}_{(u_0,Q_0)} (\uhone,\Qhone)
- \mathscr{L}^{-1}_{Q_0}\mathscr{N}_{(u_0,Q_0)}(\uhtwo,\Qhtwo) }_{X_0}
\\ &\, \qquad
  \leq  C_{\mathscr{L}}\|\mathscr{N}_{(u_0,Q_0)}
  (\uhone,\Qhone)-\mathscr{N}_{(u_0,Q_0)}(\uhtwo,\Qhtwo)\|_{Y_0}
\\ &\, \qquad
  \leq  C_{\mathscr{L}}C_{\mathscr{N} }(T,R,u_0,Q_0)
  \|(\uhone-\uhtwo,\Qhone-\Qhtwo)\|_{X_0}\,.
\end{align*}
Therefore $\mathscr{A} $ is a contraction mapping for $T\ll 1$.
A similar argument shows  that $\mathscr{A}$ maps $\BXzeroR$ into itself. In fact,
by \eqref{bounded}, we deduce that
\begin{align*}
\bnorm{ \mathscr{A} (\uhone, \Qhone) }_{X_0}
&\leq C_{\mathscr{L}} \bnorm{ \mathscr{N}_{(u_0,Q_0)}(\uhone,\Qhone) }_{Y_0}\\
&\leq C_{\mathscr{L}}\barg{ C_{\mathscr{N}}(T,R,u_0,Q_0)\|(\uhone,\Qhone)\|_{X_0}+
\|\mathcal{E}(u_0,Q_0)\|_{Y_0}}.
\end{align*}
So we can  fix $R \gg 1$ large enough
and then choose $T \ll 1$ small enough in such a way that
\begin{align*}
      \bnorm{ \mathscr{A}(\uhone,\Qhone) }_{X_0}
&\leq C_{\mathscr{L}}
C_{\mathscr{N}}(T,R,u_0,Q_0)\|(\uhone,\Qhone)\|_{X_0} + \frac R2 \leq R\,.
    \end{align*}
We conclude from Banach's fixed-point theorem that $\mathscr{A}$ possess a
unique fixed-point $(\uh,\Qh)\in X_0$ and it is
a  solution of the system \eqref{yuning:be}, according to \eqref{eq:1.06} and   Proposition \ref{lem5}.

The argument implies the uniqueness as well. Suppose that
there was another solution $(\hat{u}_h,\hat{Q}_h)$ in
$\BXzeroRone$ with $R_1>R$. Choose  $\hat{T}\leq T$ and repeat the above argument
to show the uniqueness of fixed-points of $\mathscr{A}$, which implies $(\uh,\Qh)=(\hat{u}_h,\hat{Q}_h)$
on $(0,\hat{T})\times\O$. Then the uniqueness follows by the continuity argument.

So $(u,Q)=(u_h,Q_h)+(u_0,Q_0)$ is a solution of \eqref{yuning:be} with
\begin{equation}\label{yuning:timereg}
  \begin{split}
    u&\, \in H^2(0,T;\VNSprimeO)\cap H^1(0,T;\VNSO)\,,\\
Q&\, \in H^2(0,T;L^2(\Omega;\mathcal{Q}))\cap H^1(0,T;H^2(\O;\mathcal{Q}))\,.
  \end{split}
\end{equation}
and it follows from standard interpolation result that
\begin{equation}\label{eq:1.25}
  u\in C([0,T];H^1_{0,\sigma}(\O)),\quad Q\in C([0,T];H^2(\O;\mathcal{Q}))
\end{equation}
and
\begin{equation}\label{eq:1.22}
  \begin{split}
    u_t  \in C([0,T];L^2_\sigma(\O)),\quad Q_t \in C([0,T];H^1(\Omega;\mathcal{Q})).
  \end{split}
\end{equation}
These also imply, together with Sobolev embedding that
\begin{equation}\label{eq:1.29}
  \begin{split}
    Q_t+u\cdot \nabla Q\in C([0,T];W^{1,3/2}(\O)).
  \end{split}
\end{equation}

\medskip

\textit{\underline{Step 2: Spatial regularity.}}

For any vector field $v\in H^1(\O;\R^{3})$,  we denote the symmetric matrix
\begin{equation*}
\mathcal{T}( Q ,\nabla v):=\mathcal{S}_Q ( D(v))- Q \cdot W(v)+  W(v)\cdot  Q,
\end{equation*}
where the operator $\mathcal{S}_Q$ is defined by \eqref{eq:1.15}. We also denote
\begin{equation*}
   \sigma( Q ,\nabla v):=  \mathcal{S}_Q (\mathcal{T}( Q ,\nabla v))- Q \cdot \mathcal{T}( Q ,\nabla v)+\mathcal{T}( Q ,\nabla v)\cdot Q .
\end{equation*}
Note that, for any vector field $v\in H^1(\O;\R^{3})$, not necessarily divergence-free, $\mathcal{T}( Q ,\nabla v)$ is traceless and symmetric according to \eqref{eq:1.15}.
Then we have from $Q$-tensor equation \eqref{eq:BE-Q} that
\begin{equation*}
 \mathcal{H}(Q) =-  \mathcal{T}( Q , \nabla u)+ (\partial_t Q + u\cdot\nabla Q ).
\end{equation*}
In addition, if we define
 \begin{equation}\label{eq:1.23}
  f:=-\mathcal{S}_Q(Q_t+u\cdot\nabla Q)+Q\cdot(Q_t+u\cdot\nabla Q)-(Q_t+u\cdot\nabla Q)\cdot Q,
\end{equation}
then the following identity holds:
\begin{equation*}
\begin{split}
  & D(u)+ \sigma( Q ,\nabla u)+f\\
  =& D(u)+\mathcal{S}_Q(\mathcal{T}(Q,\nabla u)-(Q_t+u\cdot\nabla Q))\\
  &+Q\cdot(Q_t+u\cdot\nabla Q-\mathcal{T}(Q,\nabla u))-(Q_t+u\cdot\nabla Q-\mathcal{T}(Q,\nabla u))\cdot Q\\
  =& D(u)-\mathcal{S}_Q(\mathcal{H}(Q))+Q\cdot \mathcal{H}(Q)-\mathcal{H}(Q)\cdot Q\\
  =&\sigma^s+\sigma^a,
\end{split}
\end{equation*}
 or equivalently
\begin{equation}\label{eq:stokes}
D(u)+ \sigma( Q ,\nabla u)+f
  =\sigma^s+\sigma^a.
\end{equation}
Substituting \eqref{eq:stokes} into \eqref{eq:BE-v} leads to
\begin{equation}\label{eq:stokes1}
  \nabla \cdot (  D(u)+\sigma(Q,\nabla u))+\nabla P=-\nabla\cdot (f+\sigma^d)+u\cdot\nabla u+u_t.
\end{equation}
If we denote
\begin{equation*}
\tilde{f}:= -\nabla\cdot (f+\sigma^d) +u\cdot\nabla u,
\end{equation*}
where $f$ is defined by \eqref{eq:1.23}, then   due to the regularity result \eqref{yuning:timereg}, we can show that
\begin{equation}\label{eq:1.33}
\tilde{f}\in C([0,T];L^{3/2}(\Omega)).
\end{equation}
Actually,  it follows from \eqref{eq:1.29}, \eqref{eq:1.15} and Sobolev embedding that
 \begin{equation}\label{eq:1.30}
   \mathcal{S}_Q(Q_t+u\cdot\nabla Q)\in C([0,T];W^{1,3/2}(\O)\cap L^3(\O))
 \end{equation}
 and also $f\in C([0,T];W^{1,3/2}(\O))$. These together with \eqref{eq:1.25} imply \eqref{eq:1.33}.

The crucial observation is that, \eqref{eq:stokes1} is  a Stokes system with variable coefficient. To show this, we claim that, the bilinear form
\begin{equation*}
 a(u,v):= \langle  D(u)+  \sigma(Q,\nabla u),\nabla v\rangle
\end{equation*}
defines a symmetric positive definite bilinear form on $H^1(\Omega;\R^3)$. Actually, note that $\mathcal{S}_Q (\mathcal{T})$ is symmetric while $ Q \cdot \mathcal{T}-\mathcal{T}\cdot Q $ is antisymmetric, we infer  from \eqref{yuning:cal1} that
\begin{equation*}
\begin{split}
\sigma( Q ,\nabla v):\nabla u &= \Big( \mathcal{S}_Q (\mathcal{T}( Q ,\nabla v))
- Q \cdot \mathcal{T}( Q ,\nabla v)+\mathcal{T}( Q ,\nabla v)\cdot Q \Big):\nabla u \\\nonumber
&=\Big( \mathcal{S}_Q(D(u))- Q \cdot W(u)+W(u)\cdot Q \Big): \mathcal{T}(Q,\nabla v) \\
&=  \mathcal{T}(Q,\nabla u) : \mathcal{T}(Q,\nabla v).
\end{split}
\end{equation*}
This formula together with the definition of $\mathcal{T}( Q ,\nabla v)$  implies that, there exists a smooth tensor-valued function
$$\left\{\hat{A}^{k\ell }_{ij}(z): \mathcal{Q}\mapsto \R\right\}_{1\leq i,j,k,\ell\leq 3}$$
with
\begin{equation*}
  \xi_i^k\hat{A}^{k\ell}_{ij}(z)\xi_j^\ell\geq 0,~\forall \xi\in \R^{3\times 3},\quad \forall z\in\mathcal{Q}
\end{equation*}
such that the following identity holds almost everywhere for $(x,t)\in\O_T$:
\begin{equation*}
  \sigma(Q(x,t),\nabla u):\nabla v = \p_k u_i\hat{A}^{k\ell}_{ij}(Q(x,t))\p_\ell v_j,\quad\forall u,v\in H^1(\O;\R^3).
\end{equation*}
Consequently, the system \eqref{eq:stokes1} can be reduced to
\begin{equation*}
  \left\{
  \begin{array}{rl}
     \pa_i \((\delta_{i}^j\delta_{k}^\ell+\hat{A}^{k\ell}_{ij}(Q(x,t)))\pa_j u_\ell\)+\p_k P&=\tilde{f}_k+\p_t u_k\\
    \nabla\cdot u&=0,\\
    u|_{\p\O}&=0.
  \end{array}
  \right.
\end{equation*}
Then it follows from Corollary \ref{lem2} as well as \eqref{eq:1.25} that\begin{equation*}
  \|u\|_{W^{2,3/2}(\Omega)}+\|\nabla P\|_{L^{3/2}(\Omega)}\leq C \(1+\|Q\|_{C([0,T];H^2(\O))}\)\|\tilde{f}+\p_t u\|_{L^{3/2}(\Omega)}, ~a.e.~t\in [0,T],
\end{equation*}
and this yields the second order derivative estimate for the velocity field:
\begin{equation}\label{eq:1.16}
  \begin{split}
  &\|u\|_{L^\infty(0,T;W^{2,3/2}(\Omega))} +\|\nabla P\|_{L^\infty(0,T;L^{3/2}(\O))}\\
  \leq &C \(1+\|Q\|_{C([0,T];H^2(\O))}\)\|\tilde{f}+\p_t u\|_{C([0,T];L^{3/2}(\Omega))}.
  \end{split}
\end{equation}
On the other hand,   we can write \eqref{eq:BE-Q}  as
\begin{align}\nonumber\label{eq:1.28}
  Q_t-\mathcal{L}(Q)&=
 -\mathcal{J}(Q)-u\cdot\nabla  Q  - Q \cdot W(u)+W(u)\cdot Q +\mathcal{S}_Q (D(u))
\\ &=: \mathcal{N}(u,\nabla u, Q, \nabla Q).
\end{align}
We claim that
\begin{equation}\label{eq:1.34}
  \mathcal{N}(u,\nabla u,Q,\nabla Q)\in L^\infty(0,T;W^{1,3/2}(\O)).
\end{equation}
Indeed,
it follows from  \eqref{eq:1.16} that $\nabla u\in L^\infty(0,T;W^{1,3/2}\cap L^3(\O))$ and this together with \eqref{eq:1.25},
Sobolev embedding and H\"{o}lder's inequality implies \eqref{eq:1.34}. Consequently, we can apply
Lemma \ref{lem:estimate-L} to deduce the higher order regularity of  $Q$:
\begin{equation}\label{eq:1.17}
\begin{split}
&\|Q\|_{L^\infty(0,T;W^{3,3/2}(\O))}\\
\leq &C \(\|\mathcal{N}(u,\nabla u,Q,\nabla Q)\|_{L^\infty(0,T;W^{1,3/2}(\O))}+\|Q_t\|_{L^\infty(0,T;W^{1,3/2}(\O))}+\|Q\|_{W^{3-2/3,3/2}(\p\O)}\).
\end{split}
\end{equation}

\medskip

 Combining  \eqref{eq:1.16}, \eqref{eq:1.17} and \eqref{yuning:timereg}, one can verify as previously that \eqref{eq:1.29}  can be improved to be
 \begin{equation*}
    Q_t+u\cdot \nabla Q\in L^\infty( 0,T ;H^{1}(\O)),
\end{equation*}
and this will in turn improve \eqref{eq:1.30} to be
 \begin{equation*}
   \mathcal{S}_Q(Q_t+u\cdot\nabla Q)\in L^\infty( 0,T ;H^1(\O)).
 \end{equation*}
 So we end up with an improved estimate    $\tilde{f}\in L^{\infty}( 0,T ;L^2(\O))$, in contrast to \eqref{eq:1.33}.
So one can argue   in the same manner as in the previous step by employing the second part of Corollary \ref{lem2}:
\begin{equation*}
  \|u\|_{L^{\infty}(0,T;H^{2}(\Omega))}   \leq C \|(\tilde{f},u_t)\|_{L^{\infty}(0,T;L^{2}(\Omega))}.
  \end{equation*}
  This implies $\mathcal{S}_Q(D(u))\in L^\infty(0,T;H^1(\O))$ and thus $ \mathcal{N}(u,\nabla u,Q,\nabla Q)\in L^\infty(0,T;H^1(\O))$. Together with \eqref{eq:1.22} and the boundary condition $Q|_{\p\O}=Q_0|_{\p\O}\in H^{5/2}(\p\O)$,  we can employ Lemma \ref{lem:estimate-L} to enhance the regularity of   \eqref{eq:1.28}  by
$Q\in L^\infty(0,T;H^3(\O))$. This completes the proof of the main theorem.
\end{proof}

\section*{Acknowledgment}
Y. Liu is supported by the Affiliated Faculty Research Grant of NYU shanghai.
W. Wang is supported by NSF of China under Grant 11501502 and ``the Fundamental Research Funds for the Central Universities" 2016QNA3004.

\end{document}